\documentclass[11pt]{amsproc}

\usepackage{a4wide}
\usepackage[pagebackref=true,colorlinks,linkcolor=red,citecolor=blue,urlcolor=blue,hypertexnames=true]{hyperref}
\usepackage{setspace}
\usepackage{amsmath}
\usepackage{amssymb}
\usepackage{amsfonts}
\usepackage{amsthm}
\usepackage{mathtools}
\usepackage{array}
\usepackage{comment, color}

\theoremstyle{plain}
\newtheorem{thm}{Theorem}[section]
\theoremstyle{definition}
\newtheorem*{defn}{Definition}
\newtheorem{lem}[thm]{Lemma}
\newtheorem{prop}[thm]{Proposition}
\newtheorem{cor}[thm]{Corollary}
\theoremstyle{remark}
\newtheorem{rem}[thm]{Remark}

\newcommand{\til}[1]{\widetilde{#1}}
\newcommand{\norm}[1]{\left\Vert#1\right\Vert}
\newcommand{\set}[1]{\left\{#1\right\}}

\newcommand{\abs}[1]{\left\lvert#1\right\rvert}

\newcommand{\eps}{\varepsilon}

\newcommand{\bbN}{\mathbb{N}}

\newcommand{\bbR}{\mathbb{R}}

\newcommand{\cB}{\mathcal{B}}
\newcommand{\cC}{\mathcal{C}}

\newcommand{\cH}{\mathcal{H}}

\newcommand{\cK}{\mathcal{K}}

\newcommand{\cM}{\mathcal{M}}

\newcommand{\cP}{\mathrm{P}}

\newcommand{\cU}{\mathrm{U}}

\newcommand{\cSU}{\mathrm{S}\mathrm{U}}

\DeclareMathOperator{\complex}{\mathrm{i}}

\DeclareMathOperator{\diag}{diag}

\title[Bounded normal generation]{Bounded normal generation for projective unitary groups of certain infinite operator algebras}

\author{Philip A. Dowerk}
\address{P.A.D., Analysis Section, KU Leuven, 3001 Leuven, Belgium}
\email{philip.dowerk@wis.kuleuven.be}

\author{Andreas Thom}
\address{A.T., Institut f\"ur Geometrie, TU Dresden, 01062 Dresden, Germany }
\email{andreas.thom@tu-dresden.de}

\begin{document}

\onehalfspace

\begin{abstract}
We study the question how quickly products of a fixed conjugacy class cover the entire group in the projective unitary group of the connected component of the identity of the Calkin algebra, as well as the projective unitary group of a factor von Neumann algebra of type III. Our result is that the number of factors that are needed is as small as permitted by the (essential) operator norm -- in analogy to a result of Liebeck-Shalev for non-abelian finite simple groups and analogous results for unitary groups of II$_1$-factors.
\end{abstract}

\maketitle

%\tableofcontents

\section{Introduction and preliminaries}

Let $G$ be a group and for $g \in G$ denote  the conjugacy class of $g$ by $g^G:= \{hgh^{-1} \mid h \in G\}$, similarly $g^{-G}:= (g^{-1})^G$. We say that $G$ has the bounded normal generation property (BNG) if for every nontrivial element $g\in G$, the conjugacy classes of $g$ and $g^{-1}$ generate the whole group in finitely many steps, i.e., there exists $k \in \mathbb N$ such that
\begin{equation}\label{def}G = (g^G \cup g^{-G})^k. \end{equation}
A function $f:G\setminus\set{1_G}\rightarrow\bbR$ giving an upper bound on the number of steps that are required is called normal generation function for $G$. Obviously, any group with property (BNG) is simple, however the converse is not true in general -- even though many naturally occuring simple groups do have property (BNG). For example, a Baire category argument (see \cite[Proposition 2.2]{DT-15}) implies that any simple compact group has property (BNG). It is much harder to provide explicit normal generating functions, even for groups like ${\rm PU}(n)$. 
Informally, we say  that a normal generating function is optimal if it is best possible up to a multiplicative constant for some family of groups.
For finite simple groups an optimal normal generation function can be found in seminal work of Liebeck-Shalev \cite[Theorem 1.1]{LS-01}. For compact connected simple Lie groups, an explicit (dimension dependent, and hence non-optimal) normal generation function was given in work of Nikolov-Segal \cite[Proposition 5.11]{NS-12}. For the finite-dimensional projective unitary groups, the authors of the present article gave a optimal normal generation function in \cite[Corollary 5.11]{DT-15}, see also \cite{ST-13}. For the projective unitary group of a type II$_1$ factor we provided a concrete (close to optimal) normal generation function in \cite[Theorem 1.3]{DT-15}. 

Note that for the groups that we are interested in here, any normal generation function is unbounded. Examples of groups with uniform bound on the normal generation function are given in \cite{GG-16,G-13,L-96,VW-95}.

In this article, we complete the picture and study the projective unitary groups $\rm PU(\cM)$ of type III factor von Neumann algebras $\cM$. Moreover, we show property (BNG) for $\rm PU_0(\mathcal{C})$, the connected component of the identity of the projective unitary group of the Calkin algebra. In both cases we provide concrete and optimal normal generation functions. For definitions and more background, we refer to \cite{DT-15}. Most techniques used in this paper are adaptions of techniques introduced in \cite{DT-15}, so that it should be read as a companion paper.

Property (BNG) for $\rm PU_0(\mathcal{C})$ (without a normal generation function) is known from the work of Fong and Sourour \cite[Theorem 1]{FS-85}. The topological simplicity of $\rm PU(\mathcal{M})$, $\cM$ a type III factor, (with the uniform topology) was discovered by Kadison \cite[Theorem 2]{Kad-52}, while the simplicity of $\rm PU(\mathcal{M})$ (for countably decomposable type III factors) is a main result in the work of de la Harpe \cite{dlH-79}.  

%Let us define the following topological version of uniform normal generators (see \cite[Definition 2.1]{DT-15} for the algebraic version).
%\begin{defn}
% Let $G$ denote a topological group. We say that $g\in G$ is a \textbf{topological uniform generator for $G$} if there exists $k\in\bbN$ such that $G=\overline{(v^G\cup v^{-G})^k}$. If every nontrivial element in $G$ is a topological uniform normal generator for $G$, then we say that $G$ has the \textbf{topological bounded normal generation property} or \textbf{property (topBNG)}.
%\end{defn}
%

The projectuve unitary group $\cP\cU(\ell^2\mathbb N)$ is not simple, but there do exist uniform normal generators, i.e., elements that satisfy Equation \eqref{def} for some $k$ -- for example symmetries having two infinite-dimensional eigenspaces, see \cite[Theorem 1]{HK-58}. Our first main result gives information on how quickly the conjugacy class of an arbitrary unitary operator and its inverse can generate another unitary.
For its proof, all necessary definitions, and in particular the definition of $d_{\rm HS}$-closure see Section \ref{sec_Calkin}. 
For a von Neumann algebra $\cM$ we define $\ell(x)\coloneqq\inf_{\lambda\in {S^1}}\norm{1-\lambda x}$ for $x\in\cM$.
Denote by $\norm{\cdot}_{\mathrm{ess}}$ the essential norm and let $\ell_{\mathrm{ess}}(x)\coloneqq\inf_{\lambda\in {S^1}}\norm{1-\lambda x}_{\mathrm{ess}}$ for $x\in B(\ell^2 \mathbb N)$. 

\begin{thm}\label{thm_main_1}
Let $G:=\cP\cU(\ell^2\mathbb N)$ denote the projective unitary group. Assume that $u,v\in G$ satisfy $\ell(u)\leq m\ell_{\rm ess}(v)$. Then 
 $$u\in \overline{(v^{G}\cup v^{-G})^{128m}}^{d_{\rm HS}}.$$
Moreover, if the elements $u$ and $v$ are diagonalizable, then we have
$$u\in(v^{G}\cup v^{-G})^{128m}.$$
\end{thm}

Note that in the converse direction, we have that if $u$ is a product of $k$ conjugates of $v$ then $\ell_{kt}(u)\leq k\ell_t(v)$ for all $t\geq 0$, see \cite[Proposition 4.9]{DT-15}. Here $\ell_t(u)$ denotes the $t$-th generalized projective $s$-number of $u$ as introduced in \cite{DT-15}.

The preceding theorem can be used to provide an optimal normal generation function for the connected component of the projective unitary group $\rm PU_0(\mathcal{C})$ of the Calkin algebra. Note that $\rm PU(\mathcal{C}) \cong \mathbb Z \times \rm PU_0(\mathcal{C})$, where the $\mathbb Z$-factor is given by the Fredholm index, so that property (BNG) for $\rm PU_0(\mathcal{C})$ is the best that one can hope for. 
In Section \ref{sec_Calkin} we derive the following collorary of Theorem \ref{thm_main_1}. 

\begin{thm}\label{thm_calkin}
 Let $G=\rm PU_0(\mathcal{C})$ denote the connected component of the identity of the projective unitary group of the Calkin algebra on a separable infinite-dimensional Hilbert space. Let $u,v\in G$ and assume that $\ell_{\mathrm{ess}}(u)\leq m\ell_{\mathrm{ess}}(v)$. Then
 $$u\in(v^G\cup v^{-G})^{32m}.$$ 
In particular, $G$ has property {\rm (BNG)} with an optimal normal generation function $f \colon G \setminus \{1\} \to \mathbb R$ given by $f(v)\coloneqq 64/\ell_{\mathrm{ess}}(v)$.
\end{thm}

Our third main result concerns type III factor von Neumann algebras and is analogous to the result about the Calkin algebra. 

\begin{thm}\label{thm_main_3}
 The projective unitary group $\rm PU(\cM)$ of a type $\rm III$ factor von Neumann algebra has property {\rm (BNG)}. An optimal normal generation function is given by $f(v)\coloneqq 2048/\ell(v)$ for $v\in\rm PU(\cM)\setminus\set{1}$. 
\end{thm}

This result completes the picture for property (BNG) regarding von Neumann algebras: $\mathrm{PU}(\cM)$ has property (BNG) if $\cM$ is a factor of type $\mathrm{I}_n$, $\mathrm{II}_1$ or type $\mathrm{III}$. In the cases $\mathrm{I}_\infty$ and $\mathrm{II}_\infty$, the group does not have property (BNG).

\section{Generation for unitary operators on a Hilbert space}\label{sec_Calkin}

Let $\cM$ be a separable type $\mathrm{I}_{\infty}$-factor. Then $\cM=\cB(\ell^2\mathbb N)$  by Corollary V.1.28 in \cite{Tak}. In this section, we consider $G\coloneqq\cP\cU(\cM):=\cU(\cM)/(S^1 \cdot 1)$ endowed with the quotient topology of the strong operator topology. Note that $G$ is a Polish group with this topology.

%However, we will consider Hilbert-Schmidt perturbations in this group in the proof of Theorem \ref{thm_I}. Note that in the topology induced from the Hilbert-Schmidt norm, $\cU(\cM)$ respectively $\cP\cU(\cM)$ is not a topological group. 
%
%We explain the notion used in Theorem \ref{thm_I}.
%Let $u,v\in G$, $n\in\bbN$ and denote by $d_{\rm HS}$ the Hilbert-Schmidt norm. The notion $u\in\overline{(v^{G}\cup v^{-G})^{n}}^{d_{\rm HS}}$ used in Theorem \ref{thm_I} means that $u$ lies in $(v^{G}\cup v^{-G})^{n}$ up to an arbitrarily $d_{\rm HS}$-small Hilbert-Schmidt perturbation. \\

There is an obstruction for the bounded normal generation property of $\cP\cU(\ell^2 \mathbb N)$ which we will now describe.
Let $K$ denote the norm-ideal of compact operators $\cK(\ell^2 \mathbb N)$ on $\ell^2 \mathbb N$, endowed with the operator norm. 

The essential norm $\norm{\cdot}_{\mathrm{ess}}$ on the Calkin algebra $\cC:= B(\ell^2 \mathbb N)/K$ is the quotient norm. For $u\in\cP\cU(\mathcal{C})$ we let $$\ell_{\mathrm{ess}}(u)\coloneqq\inf_{\lambda\in S^1}\norm{1-\lambda u}_{\mathrm{ess}}.$$ We consider $\ell_{\rm ess}$ also for unitary operators $\cU(\ell^2 \mathbb N)$ by abuse of notation.
We define $\cU(\ell^2 \mathbb N)_K$ as the group of unitary operators on a Hilbert space $\cH$ such that $1-u$ is an element of $K$, 
$$\cU(\ell^2 \mathbb N)_K\coloneqq \set{u\in\cU(\ell^2 \mathbb N) \mid 1-u\in K}.$$
The group $\cU(\ell^2 \mathbb N)_K$ is naturally endowed with the topology given by the operator norm and a polish group with this topology. 

Observe that $\cU(\ell^2 \mathbb N)_K$ is a normal subgroup of $\cU(\ell^2 \mathbb N)$ which contains every finite-dimensional unitary group $\cU(n),\ n\in\bbN$. In fact, every proper normal subgroup of $\cU(\ell^2 \mathbb N)$ is contained in $\cU(\ell^2 \mathbb N)_K$ by \cite[Theorem 1]{FS-85}. Note that the center of $\cU(\ell^2 \mathbb N)_K$ is trivial, so that $\cU(\ell^2 \mathbb N)_K$ is actually also a subgroup of the projective unitary group $\cP\cU(\ell^2 \mathbb N)$.

Similar to $\cU(\ell^2 \mathbb N)_K$, one can use other operator ideals in this contruction. A special role is played by $\cU(\ell^2 \mathbb N)_{\rm HS}$, where ${\rm HS}$ denotes the ideal of Hilbert-Schmidt operators on $\ell^2 \mathbb N$. Again, we have $\cU(\ell^2 \mathbb N)_{\rm HS} \subset \cP\cU(\ell^2 \mathbb N)$ as a normal subgroup.

We now define the Hilbert-Schmidt metric on $\cP\cU(\mathbb \ell^2 \mathbb N)$ (or $\cU(\ell^2 \mathbb N))$ by
$$d_{\rm HS}(u,v) := \begin{cases} 1 & 1-uv^{-1} \not \in \rm HS\\ %\cU(\ell^2 \mathbb N)_{\rm HS} \\
\min\{1,\|1-uv^{-1}\|_{\rm HS}\} & 1-uv^{-1} \in \rm HS.\end{cases}$$ %\cU(\ell^2 \mathbb N)_{\rm HS}. \end{cases}$$
Note that the Hilbert-Schmidt metric is much finer that the uniform metric. This metric is rarely studied, however on major result of Voiculescu \cite[Corollary 2.6 and Theorem 4.2]{V-79} in the perturbation theory of unitary operators asserts that the $d_{\rm HS}$-closure of the set of diagonalizable unitary operators contains all unitary operators. This will be of great use for us, when we need to generate unitary operators in the Calkin algebra.

%Observe that for any $u\in\cU(\ell^2 \mathbb N)_K$ we have $\ell_t(u)\leq \mu_t(1-u)\rightarrow 0$ for $t\rightarrow \infty$ by compactness of $1-u$. In particular, for elements $u,v\in\cU(\ell^2 \mathbb N)_K$ there usually does not exist a number $m\in\bbN$ such that $\ell_0(u)\leq m\ell_t(v)$ for all $t\geq 0$. This is the obstruction for the bounded normal generation property of $\cP\cU(\ell^2 \mathbb N)$. As we will see, all problems disappear when passing to the projective unitary group of the Calkin algebra.
%

We embed $u\in\cU(n)$ into $\cU(\ell^2 \mathbb N)_K$ in the usual way by $\cU(n)\ni u\mapsto\left(\begin{smallmatrix}u&0\\0 &1\end{smallmatrix}\right)\in\cU(\ell^2 \mathbb N)_K$. It is not hard to show that the unions $\bigcup_{n\in\bbN}\cU(n)$ as well as $\bigcup_{n\in\bbN}\cSU(n)$ are dense in $\cU(\ell^2 \mathbb N)_K$ in the uniform topology. 

%It is known that $\cU(\ell^2 \mathbb N)_K$ is topologically simple in the uniform topology. 
%However, there is no topological uniform normal generator for $\cU(\ell^2 \mathbb N)_K$. Suppose the contrary and let $v$ be a topological uniform normal generator for $\cU(\ell^2 \mathbb N)_K$. Then one can replace the sequence of singular values of $1-v\in K$ with their square roots and obtain a corresponding element $u\in\cU(\ell^2 \mathbb N)_K$. But then there exists no $k\in\bbN$ such that $\ell_{kt}(u)\leq k\ell_t(v)$ for all $t\geq 0$, which contradicts \cite[Proposition 4.9]{DT-15}.

In anology to the finite-dimensional case (see \cite[Section 5]{DT-15}) we deal with two different decompositions of diagonal unitary operators $u \in \cU(\ell^2 \mathbb N)$ as product of simpler operators. If $u=\diag(\lambda_0,\lambda_1,\ldots)\in\cU(\ell^2 \mathbb N)$ is  diagonal, then we can write $u=\prod_{j\in\bbN_0}u_j$ in the \textbf{product decomposition} with $u_j$ defined as follows:
$$u_j=\diag(1,\ldots,1,\lambda_0\cdot\ldots\cdot\lambda_j,\overline{\lambda}_0\cdot\ldots\cdot\overline{\lambda}_j,1,1,\ldots),\ j\in\bbN_0.$$
Then, we have 
$$\prod_{j=0}^nu_j=\diag(\lambda_0,\ldots,\lambda_n,\overline{\lambda}_0\cdot\ldots\cdot\overline{\lambda}_n,1,1,\ldots)\rightarrow_{n\rightarrow \infty} u$$
strongly, but in general not uniformly.
We also need the \textbf{torus decomposition} $\prod_{j\in\bbN_0}\til{u}_j$ of $u$ by setting 
$$\til{u}_0=\diag(\lambda_0,\lambda_0,\ldots),\ \til{u}_j=\diag(1,\ldots,1,\overline{\lambda}_{j-1}\lambda_j,\overline{\lambda}_{j-1}\lambda_j,\ldots), \ j\in\bbN_0.$$
Then, in the strong operator topology, but in general not in the uniform topology, we have
$$\prod_{j=0}^n \til{u}_j=(\lambda_0,\lambda_1,\ldots,\lambda_{n-1},\lambda_n,\lambda_n,\ldots)\rightarrow_{n\rightarrow\infty} u.$$
%That is, when $\cU(\ell^2 \mathbb N)$ is endowed with the strong operator topology, then the above elements coincide: 
%$$u'=\prod_{j\in\bbN_0}\til{u}_j=\prod_{j\in\bbN_0}u_j.$$

Recall the following lemma, that will be useful for our generation processes.
\begin{lem}[Nikolov-Segal]\label{lem_su2}
 Let $u=\left(\begin{smallmatrix}e^{\complex\varphi} &0\\ 0 &e^{-\complex\varphi}\end{smallmatrix}\right)$ and $v=\left(\begin{smallmatrix}e^{\complex\theta} &0\\ 0 &e^{-\complex\theta}\end{smallmatrix}\right)$ be non-central elements in $G\coloneqq \cSU(2)$ with $\theta \in [-\pi/2,\pi/2]$. If $\abs{\varphi}\leq m\abs{\theta}$ for some even $m\in\bbN$, then $u\in(v^G)^m$. 
\end{lem}

We are now ready to prove the infinite-dimensional analogue of \cite[Lemma 5.3]{DT-15}.
As in \cite[Section 5]{DT-15} we denote by $S_j$ a copy of $\cSU(2)$ embedded in $\cU(\ell^2 \mathbb N)$ at the diagonal entries $j,j+1$.

\begin{lem}\label{lem_infsim}
 Consider $u=\diag(e^{\complex\theta_0},e^{\complex\theta_1},\ldots),v=\diag(e^{\complex\gamma_0},e^{\complex\gamma_1},\ldots)\in G\coloneqq \cU(\ell^2 \mathbb N)$. Assume that $\prod_{i\in\bbN_0}u_i$ is the product decomposition of $u$ and $v=\prod_{i\in\bbN_0}v_i$ is the torus decomposition of $v$. 
 %Let $u_{i_k}$ and $v_{j_l}$ correspond to $u=\diag(e^{\complex\theta_0},\ldots)$ and $v=\diag(e^{\complex\gamma_0},\ldots)$ as above such that $\abs{i_k-i_l},\abs{j_k-j_l}>1$ for all $k\neq l$ and some countable index sets $I, J\subseteq \bbN$ and $\abs{\theta_{i_k}-\theta_{i_k+1}}\leq m\abs{\gamma_{j_k}-\gamma_{j_k+1}}$ for some even number $m\in\bbN$, $k,l=0,1,\ldots$. 
Let $I=\set{i_0,i_1,\ldots},J=\set{j_0,j_1,\ldots}\subseteq\bbN_0$ be countable index sets such that $\abs{i_k-i_l}>1$ and $\abs{j_k-j_l}>1$ for all $k\neq l\in\bbN_0$, where $i_k,i_l\in I,j_k,j_l\in J$. If 
\begin{equation} \label{e1}
\abs{\sum^{i_k}_{i=1}\theta_{i}}\leq m\abs{e^{\complex \gamma_{j_k}}-e^{\complex\gamma_{j_k+1}}}
\end{equation} for some even number $m\in\bbN$ and all $k\in\bbN_0$. 
 Then $$\prod_{i\in I}u_{i}\in \left(v^G\cup v^{-G}\right)^{4m}.$$
\end{lem}
\begin{proof}
 Write as above $u=\til{u}\prod_{k=1}^{\infty}u_{i_k}$ where $\til{u}=\prod_{l\in\bbN\setminus I}u_l$ and analogously $v=\til{v}\prod_{l=1}^{\infty}v_{j_l}$. Note that $S_{j_k}$ and $S_{j_l}$ commute elementwise for $k\neq l$, $j_k,j_l\in J$. Moreover, $\til{v}$ commutes with $S_{j_k}$ for all $j_k\in J$. Thus we get 
 $$ \left(v_{}^{\prod_{j\in J}S_{j}}\right)^m=\left(\til{v}_{}^{\prod_{j\in J}S_{j}} \prod_{j\in J}v_{j}^{S_{j}}\right)^m= \til{v}_{}^m \prod_{j\in J}\left(v_{j}^{S_{j}}\right)^m.$$
 Let $w\in\cU(\ell^2 \mathbb N)$ be an unitary operator such that $S_{i_k}^w=S_{j_k}$ for all $k\in\bbN_0$. Using Lemma \ref{lem_su2} and the inequality in \eqref{e1}, we obtain $u_{i_k}^w\in\left(v_{j_k}^{S_{j_k}}\right)^{2m}$ for all $k\in\bbN_0$, and hence
 $$\prod_{i\in I}u_{i}\in \left(\prod_{k \in \bbN_0} \left(v_{j_k}^{S_{j_k}}\right)^{2m} \right)^{w^*} = \left(\left( \left(\prod_{k \in \bbN_0} v_{j_k}\right)^{\prod_{k \in \bbN_0}{S_{j_k}}} \right)^{2m} \right)^{w^*}  = \left(\til{v}^{-2m}\left(v^{\prod_{j\in J}S_{j}}\right)^{2m}\right)^{w^{*}}.$$
 Since 
 $$(\til{v})^{-2}\in \til{v}^{-2}\prod_{j\in J}\left(v_{j}^{S_{j}}\right)^{-2}=\left(v^{\prod_{j\in J}S_{j}}\right)^{-2},$$
 we conclude
 $$\prod_{i\in I}u_{i}\in\left(\left(v^{\prod_{j\in J}S_{j}}\right)^{2m}\cdot\left(v^{-\prod_{j\in J}S_{j}}\right)^{2m}\right)^{w^{*}}\subset \left( v^G\cup v^{-G}\right)^{4m}.$$
 This finishes the proof.
\end{proof}

Before being able to define an infinite-dimensional variant of angle-sum optimality (cf. \cite[Definition 5.8]{DT-15}), we need to generalize \cite[Lemma 5.7]{DT-15} to our situation. 
\begin{lem}\label{lem_sum_angles_inf}
 Let $\set{\alpha_n}_{n\in\bbN_0}\subseteq \bbR$ be a sequence with infinitely many positive and infinitely many negative real numbers such that both the sum over all positive $\alpha_n$ and over all negative $\alpha_n$ diverge. Then there exists a permutation $\sigma\in S_{\infty}$ such that 
$$\abs{\sum_{i=0}^n \alpha_{\sigma(i)}}\leq \sup_{j\in\bbN_0}\abs{\alpha_j}\ \mbox{ for every } n\in\bbN_0.$$
\end{lem}
\begin{proof}
We define $\sigma$ inductively. If $\sigma$ is already defined on $\{0,\dots,n\}$, we define $\sigma(n+1)$ to be the next index $k$ where $\alpha_k \leq 0$ if 
$\sum_{i=0}^n\alpha_{\sigma(i)}>0$ and the next index $k$ where $\alpha_k \geq 0$ if $\sum_{i=0}^n\alpha_{\sigma(i)} \leq0.$
\end{proof}

\begin{defn}
 Let $u=\diag(e^{\complex\theta_0},e^{\complex\theta_1},\ldots)$ with $\theta_i \in (\pi,\pi]$. Then we say that $u$ is \textbf{angle sum ordered} if the sequence $(\theta_n)_{n}$ satisfies
$$\abs{\sum_{i=0}^n \theta_{i}}\leq \sup_{j\in\bbN_0}\abs{\theta_j}\ \mbox{ for every } n\in\bbN_0.$$
\end{defn}

%Let $u\in \cU(\ell^2 \mathbb N)\setminus \cU(\ell^2 \mathbb N)_K$, $u\neq 1$. Let us interpret the points in the spectrum as elements $e^{\complex\varphi}\in S^1$ with angle $\varphi\in(-\pi,\pi]$. Observe that there is $\lambda\in S^1$ such that the essential spectrum of $\lambda u$ meets the upper and the lower half of $S^1$.
%For example, if the spectrum $\sigma(u)$ of $u$ is $\set{e^{\complex\pi/2},1}$, where $e^{\complex\pi/2}$ and $1$ have infinite multiplicities, then one can choose $\lambda=e^{-\complex\pi/4}$ to obtain $\sigma(\lambda u)=\set{e^{\complex\pi/4},e^{-\complex\pi/4}}$.
%
%
%\in \cU(\ell^2 \mathbb N)\setminus \cU(\ell^2 \mathbb N)_K$, $u\neq 1$, and $\sigma\in S_{\infty}$ as in Lemma \ref{lem_sum_angles_inf}. 
%

\begin{lem} \label{angleopt}
For any diagonalizable unitary $u\in\cU(\ell^2 \mathbb N) \setminus (S^1 \cdot \cU(\ell^2 \mathbb N)_K)$, there exists $\lambda \in S^1$ such that some element in $\diag(e^{\complex\theta_0},e^{\complex\theta_1},\ldots) \in (\lambda u)^G$, where $G\coloneqq \cU(\ell^2 \mathbb N)$, is angle sum ordered. Moreover, we may assume that
$$\abs{\sum_{i=0}^n\theta_{\sigma(i)}}\leq 2\ell(u).$$
\end{lem}
\begin{proof} If $u \not \in S^1 \cdot \cU(\ell^2 \mathbb N)_K$ we may rotate $u$ so that it has essential spectrum on the upper and lower half of the unit circle. The claim follows now from Lemma \ref{lem_sum_angles_inf}. \end{proof}

We are now ready to prove Theorem \ref{thm_main_1}.
\begin{proof}[Proof of Theorem \ref{thm_main_1}] If $u=1 \in \cP\cU(\ell^2 \mathbb(\mathbb N))$, then there is nothing to prove. Hence, we may assume that $\ell(u)>0$ and consider some lift of $u$ to $\cU(\ell^2 \mathbb N)$. There exists $\delta>0$ such that 
$\ell(u)\leq 2m\ell_{\rm ess}(v)-\delta$.
We choose $\eps>0$ such that $\eps\leq\delta$.
Using a version of the non-commutative Weyl-von Neumann theorem by Voiculescu, see \cite[Corollary 2.6 and Theorem 4.2]{V-79}, we obtain the existence of diagonal elements 
$$u'=\diag(e^{\complex\theta_0},e^{\complex\theta_1},\ldots),\quad v'=\diag(e^{\complex\gamma_0},e^{\complex\gamma_1},\ldots)$$
and $g,h\in G$ such that $gu'g^{-1}$ and $hv'h^{-1}$ are $\eps$-close to $u$ and $v$ in the Hilbert-Schmidt metric $d_{\rm HS}$ respectively. 
%By Lemma \ref{lem_singrel} and since the Hilbert-Schmidt norm is always greater or equal than the operator norm, we have 
Since the Hilbert-Schmidt norm is always greater or equal than the operator norm, we have 
%$$ \ell(u')\leq 2m\ell_{\rm ess}(v')-\delta+2(m+1)\eps  \leq 2m\ell_{\rm ess}(v').$$
$$ \ell(u')\leq \ell(u)+\norm{u-u'}\leq 2m\ell_{\rm ess}(v')-\delta+\eps  \leq 2m\ell_{\rm ess}(v').$$
From the inequality $2m\ell_{\rm ess}(v') \geq \ell(u')$
%, from $\ell_t(v')\geq\ell_{\rm ess}(v')$ 
and from \cite[Lemma 5.5]{DT-15} we conclude the existence of a set $J=\set{j_0,j_1,\ldots}\subseteq\bbN_0$ such that $\abs{j_k-j_l}>1$ for all $k\neq l\in\bbN_0$, where $j_k,j_l\in J$, $v''\coloneqq\diag(e^{\complex\gamma_{j_0}},e^{\complex\gamma_{j_1}},\ldots)\in (v')^G$ and
$$\ell(u') \leq 2m \abs{e^{\complex \gamma_{j_k}}-e^{\complex\gamma_{j_k+1}}}
.$$

Assume now that $u \not \in S^1 \cdot \cU(\ell^2 \mathbb N)_K$ and $u''\in(u')^G$ is angle sum ordered (with corresponding permutation $\sigma$). Combining our results with Lemma \ref{angleopt}, we obtain
\begin{equation}\label{e2}
 \abs{\sum_{i=0}^n\theta_{\sigma(i)}}\leq 2\ell(u'') \leq 4m \cdot \abs{e^{\complex \gamma_{j_k}}-e^{\complex\gamma_{j_k+1}}}
 \end{equation}
for all $n,k \in \mathbb N$.
Before being able to generate $u''$ from $v''$, we need to decompose both elements in an appropriate way. Decompose $u''=\prod_{i\in\bbN_0}u_i$ into its product decomposition and $v''=\prod_{i\in\bbN_0}v_i$ into its torus decomposition. In order to generate infinitely many entries of $u''$ simultaneously, we partition $\bbN_0$ into disjoint sets $A_1$ and $A_2$, where $A_1\coloneqq\set{2n\mid n\in\bbN_0}$ and $A_2\coloneqq\set{2n+1\mid n\in\bbN_0}$.\\
 We use Lemma \ref{lem_infsim} and the inequality in \eqref{e2} to obtain
$\prod_{j\in A_i}u_j\in \left((v'')^G\cup (v'')^{-G}\right)^{16m}$ for $i=1,2$
and thus 
 $$u''\in \left((v'')^G\cup (v'')^{-G}\right)^{32m}.$$
 Since $u''\in (u')^G$, $v''\in (v')^G$ and $u'\in(u^G)_{\eps,d_{\rm HS}}$, $v'\in(v^G)_{\eps,d_{\rm HS}}$, we get
 $$u\in\left((u'')^G\right)_{\eps}\subseteq \left(\left( (v^G)_{\eps,d_{\rm HS}}\cup(v^{-G})_{\eps,d_{\rm HS}}\right)^{32m}\right)_{\eps,d_{\rm HS}}.$$
% Observe that if $u$ and $v$ were diagonal from the beginning, then $u'=u,\ v'=v$ and we are already done. In the general case, we have $u'\in(u^G)_{\eps,d_{\rm HS}}$, $v'\in(v^G)_{\eps,d_{\rm HS}}$ and thus 
% $$u\in\left(\left( (v^G)_{\eps,d_{\rm HS}}\cup(v^{-G})_{\eps,d_{\rm HS}}\right)^{1500km}\right)_{\eps,d_{\rm HS}}.$$
 Thus by \cite[Lemma 2.5]{DT-15} we have $$u\in\left(\left( v^G\cup v^{-G}\right)^{32m}\right)_{(32m+1)\eps,d_{\rm HS}}.$$
 Letting $\eps$ tend to zero, we arrive at the conclusion 
 $$u\in\overline{\left(v^G\cup v^{-G}\right)^{32m}}^{d_{\rm HS}}.$$
In case 
$u' \in S^1 \cdot \cU(\ell^2 \mathbb N)_K$ is diagonalizable we may assume that $1$ is the only cluster point in the spectrum, i.e., $u' \in \cU(\ell^2 \mathbb N)_K$. By a Hilbert-Schmidt perturbation, we may restrict to the case that $1-u'$ is not trace class, i.e., $u'= \diag(e^{\complex\theta_0},e^{\complex\theta_1},\ldots)$ with $\sum_i |\theta_i| = \infty$. We can easily define sequences $(\theta'_n)$ and $(\theta''_n)$ such that $\theta_n = \theta'_n + \theta''_n$, $\max\{ |\theta'_n|, |\theta''_n| \} \leq 2|\theta_n|$ so that the sums of the positive elements and the sums over negative elements in each of the sequences diverge. In this case then, we may find an angle ordered conjugate and argue as before -- changing the constant from $32$ to $128$. This completes the proof.
\end{proof}

\begin{rem}
 Fong and Sourour showed in \cite{FS-85} that every proper normal subgroup of $\cU(\ell^2 \mathbb N)$ is contained in the normal subgroup $\cU(\ell^2 \mathbb N)_{K}$ and that if $v\in\cU(\ell^2 \mathbb N)\setminus \cU(\ell^2 \mathbb N)_{K}$, then each $u \in \cU(\ell^2 \mathbb N)$ is a product of a finite number of conjugates of $v$. So Theorem \ref{thm_main_1} on the one hand can be considered weaker than the result of Fong and Sourour because it involves the Hilbert-Schmidt norm closure, but on the other hand we give quantitative estimates. Our version will allow us to prove property (BNG) with an explicit normal generation function for the connected component of the identity of the projective unitary group of the Calkin algebra.
\end{rem}

Let us now concentrate on the Calkin algebra. We write $\cU(\mathcal{C})$ for its unitary group and $\cP\cU(\mathcal{C})$ for its projective unitary group. Each equivalence class in $\cU(\mathcal{C})$ contains a diagonal element by the Weyl-von Neumann-Berg-Voiculescu Theorem, see \cite[Corollary 2.6 and Theorem 4.2]{V-79}. 
By \cite[Theorem 4.1.6]{Mur} $\mathcal{C}$ is a simple $C^*$-algebra.
%Let us explain why we restrict our attention in Theorem \ref{thm_calkin} to the connected component $G$ of $\cP\cU(\mathcal{C})$ of the identity. 
The (projective) unitary group of the Calkin algebra is not connected (recall that in contrast, the unitary group of a von Neumann algebra is always connected in the uniform topology and hence also in the strong operator topology, see \cite[Exercise 5.7.24(ii)]{KR}). Its connected components are characterized by the Fredholm index. Since we want to use Theorem \ref{thm_main_1} we need to ensure that the elements of consideration in $\cP\cU(\mathcal{C})$ can be lifted to $\cU(\ell^2 \mathbb N)$. This lift exists precisely if the elements have Fredholm index $0$, that is, they are in the connected component of the identity.\\

%Now let $u,v\in\cP\cU(\cC)\setminus\set{1}$ be nontrivial diagonal elements of Fredholm index 0 satisfying 
%$$\inf_{\lambda\in\cZ(\cU(\ell^2 \mathbb N))}\norm{1-\lambda u}_{ess}\leq m\inf_{\lambda\in\cZ(\cU(\ell^2 \mathbb N))}\norm{1-\lambda v}_{ess},$$
%where $\norm{\cdot}_{ess}$ denotes the essential norm on $\cC$. We may write $u=\diag(e^{\complex\theta_0},e^{\complex\theta_1},\ldots)$ and $v=\diag(e^{\complex\gamma_0},e^{\complex\gamma_1},\ldots)$ for some $\theta_j,\gamma_j\in[0,2\pi]$. By assumption, there are at least two different $\gamma_j$, let us say $\gamma_0$ and $\gamma_1$ form such a pair. The corresponding eigenvalues have infinite multiplicity.

%
%\begin{thm}\label{thm_calkin}
% Let $G$ denote connected component of the identity of $\cP\cU(\mathcal{C})$.
% Assume that $u,v\in G\setminus\set{1}$ satisfy $\ell_{\mathrm{ess}}(u)\leq m\ell_{\mathrm{ess}}(v)$. Then, $u\in (v^G\cup v^{-G})^{80m}.$
% In particular, $G$ has property (BNG): if $\ell_{\mathrm{ess}}(v)>0$, then 
% $G=(v^G\cup v^{-G})^{m},$
% for every $m\geq 40/\ell_{\mathrm{ess}}(v)$. Thus, $f(v)=40/\ell_{\rm ess}(v)$ is an optimal normal generating function.
%\end{thm}

\begin{proof}[Proof of Theorem \ref{thm_calkin}]
 Let $H\coloneqq\cP\cU(\ell^2 \mathbb N)$. Since $u$ and $v$ are of Fredholm index $0$, there exists a lift into $H$. We denote the corresponding elements by $u'$ and $v'$. Since the eigenvalues of $u,v$ have infinite multiplicity, we have $\ell(u')\leq m\ell_{\mathrm{ess}}(v')$.  Hence by the proof of Theorem \ref{thm_main_1} we have 
$$u'\in\overline{(v'^H\cup v'^{-H})^{32m}}^{d_{\rm HS}}.$$
We may pass back to $\cP\cU(\cC)$ by using the quotient map, so that we obtain
$$u\in (v^G\cup v^{-G})^{32m},$$
as claimed.
To see the second claim in the statement of Theorem \ref{thm_calkin}, note that $\ell_{\mathrm{ess}}(u)\leq 2$ for any $u\in G$.
\end{proof}

Immediately, we obtain as a corollary that the connected component $G$ of the identity in $\cP\cU(\mathcal{C})$ is algebraically simple, a result that was also found by Fong and Sourour in \cite{FS-85}.

%Fong and Sourour actually showed more, see \cite[Theorem 3]{FS-85}: the normal subgroups of $\cU(\mathcal{C})$ are its center and the groups 
%$$N_n\coloneqq \set{u\in\cU(\mathcal{C})\mid n\mbox{ divides the Fredholm index of }u},\qquad n\in\bbN_0.$$
%
%
%
\section{Bounded normal generation for type III factors}

In this section we show that the projective unitary group $\mathrm{PU}(\cM)$ of a type $\rm III$ factor $\cM$ has property (BNG), see Theorem \ref{thm_main_3}. To accomplish our goal we use some ideas and results from \cite[Section 7]{DT-15}. 

We put 
$$\ell(x)\coloneqq\inf_{\lambda\in S^1}\norm{1-\lambda x}\quad\mbox{ for }x\in\cM.$$
The proof of the following statement borrows ideas from \cite[Proposition 7.2]{DT-15}. It will be useful in the proof of Theorem \ref{thm_main_3} to keep track of the length function $\ell(\cdot)$ under taking appropriate commutators.
\begin{prop}\label{prop_comm}
 Let $\cM$ be a type \rm III factor and denote by $G$ its unitary group. For every $u\in G$ there exists $v\in G$ such that $\ell(u)\leq 4\ell([u,v])$.
\end{prop}
\begin{proof}
 If $u\in G$ is central, then the claim is trivial. So assume $u$ to be non-central. Approximate $u$ by $u'$ with finite spectrum in the uniform topology using functional calculus, i.e., find $\varepsilon>0$ such that $9\eps\leq \ell(u')$, $\norm{u-u'}<\eps$ and $u'=\sum_{i=0}^n \lambda_ip_i$ for some $\lambda_i\in S^1$ and projections $p_i\in\cM$. Without loss of generality, the diagonal matrix $\til{u}=\diag(\lambda_0,\lambda_1,\ldots,\lambda_n)$ is such that $\abs{\lambda_0-\lambda_1}$ is maximal among all differences $\abs{\lambda_i-\lambda_j}$. Write $p_0=p_0'+p_0''$ for some nonzero projections (equivalent to $p_0$) and analogously $p_1=p_1'+p_1''$ so that $\til{u}=\diag(\lambda_0',\lambda_0'',\lambda_1',\lambda_1'',\ldots,\lambda_n)$ with $\lambda_i'=\lambda_i''=\lambda_i$ for $i=0,1$. Let $\sigma$ be a permutation acting nontrivially only on the first four entries of $\til{u}$ as follows:
 $$\lambda_{\sigma(0)}=\lambda_1',\quad \lambda_{\sigma(1)}=\lambda_0',\quad \lambda_{\sigma(2)}=\lambda_0'',\quad \lambda_{\sigma(3)}=\lambda_1''.$$
 Let $v\in\rm U(\cM)$ such that $vp_iv^{-1}=p_{\sigma^{-1}(i)}$. Then we have
 $$[u',v]=\sum_{i=0,\ldots,n-1}\lambda_ip_i\sum_{j=0,\ldots,n-1}\overline{\lambda}_jp_{\sigma^{-1}(j)}=\sum_{i=0,\ldots,n-1}\lambda_i\overline{\lambda}_{\sigma(i)}p_i.$$
 Observe that $\ell([u,v])\geq \ell([u',v])-2\eps$. Now the estimates   
 $$4\ell([u,v])\geq 4\ell([u',v])-8\eps\geq 2\abs{\lambda_0-\lambda_1}-8\eps\geq 2\ell(u')-8\eps\geq \ell(u') +\varepsilon \geq \ell(u)$$
 imply the claim.
\end{proof}

We are now ready to prove Theorem \ref{thm_main_3}, following a strategy that was already used in our previous work, see \cite[Theorem 7.5]{DT-15}.

\begin{proof}[Proof of Theorem \ref{thm_main_3}]
 Suppose first that $\cM$ is separable.
 We will generate a nontrivial symmetry with conjugates of $v$ and $v^{-1}$ which can be used to generate the whole group $G\coloneqq\rm PU(\cM)$ by Fillmore's corollary in \cite{Fill-66}. \\ 
 Using functional calculus we can approximate $v$ by some $v'=\sum_{i=1}^n\lambda_ip_i$, $\lambda_i\in S^1,\ p_i\in\rm Proj(\cM)$ spectral projections of $v$, up to arbitrarily small $\eps>0$ in the operator norm. In particular, we can approximate such that $\ell(v)\leq 2\ell(v')$. Write $p_i=p_i'+p_i''$ for nontrivial projections $p_i'$ and $p_i''$, i.e. $v'=\sum_{i=1}^n\lambda_ip_i'+\sum_{i=1}^n\lambda_ip_i''$ and $v=v_0+v_1=vp'+vp''$, where $p'=\sum_{i=1}^n p_i'$ and $p''=\sum_{i=1}^n p_i''$. By Proposition \ref{prop_comm} we can find $w=w_0+p''\in G$, such that $[v,w]=[v_0,w_0]+p''\cong \left(\begin{smallmatrix} [v_0,w_0] & 0 \\ 0 & 1 \end{smallmatrix} \right)$ and $\ell(v)\leq 2\ell(v')=2\ell(v'p')\leq 8\ell([v_0,w_0])$. 
 
 Let $$v''\coloneqq \left( \begin{smallmatrix} [v_0,w_0] & 0 \\ 0 & [v_0,w_0]^{-1} \end{smallmatrix} \right)\in (v^G\cup v^{-G})^4.$$ 
 We pass to $(M_{2\times 2}(L^{\infty}(\sigma(v''),\mu)),\norm{\cdot})$ and observe that there exists a subset $X\subseteq \sigma(v'')$ with $\mu(X)>0$ where $\ell(p_X v'')\geq \ell(v'')/2\geq \ell(v)/16$, $p_X$ being the spectral projection corresponding to $X$.  
 Using \cite[Lemma 7.4]{DT-15} (and that $p_X$ is nonzero and thus equivalent to any nonzero projection in $\cM$) we generate a nontrivial symmetry $s$ with $8/\ell(p_X v'')\leq 128/\ell(v)$ conjugates of $v''$ (here an additional factor $4$ comes from the relation between $\ell(\cdot)$ and angles as given in \cite[Corollary 5.6]{DT-15}). Thus $s\in (v''^G\cup v''^{-G})^{\lceil 128/\ell(v)\rceil}\subseteq (v^G\cup v^{-G})^{\lceil 512/\ell(v)\rceil}.$
 Since all nontrivial symmetries in $G$ are conjugate, using the corollary in \cite{Fill-66} we conclude that every unitary can be written as the product of at most four conjugates of $s$, that is
 $$G=(v^{G}\cup v^{-G})^{\lceil 2048/\ell(v)\rceil}.$$
 If $\cM$ is not separable, then the result follows from the fact that any two unitaries in $\cM$ lie in a common separable type III factor.
\end{proof}

\begin{cor}
 The projective unitary group of a type III factor is simple.
\end{cor}

%
%\begin{rem}
% Note that a homomorphism from the projective unitary group of a type $\rm III$ factor $\cM$ to any separable SIN group is trivial. 
% It follows that the projective unitary group of a type III factor does not admit a Polish SIN group topology. 
%\end{rem}

\section*{Acknowledgments}

The research leading to these results has received funding from the European Research Council under the European Union's Seventh Framework Programme ERC StG 277728 and  ERC CoG 614195. P.A.D.\ wants to thank Universit\"at Leipzig, the IMPRS Leipzig and the MPI-MIS Leipzig for support and an stimulating research environment. The material on the Calkin algebra appeared as part of the PhD-thesis of the first author.

\bibliographystyle{alpha}

\end{document}